\documentclass[12pt]{scrartcl}

\usepackage{amsthm, amsmath, amssymb}
\usepackage[alphabetic]{amsrefs}
\usepackage{graphicx}
\allowdisplaybreaks

\newtheorem{theorem}{Theorem}[section]

\newtheorem{lemma}[theorem]{Lemma}
\newtheorem{definition}[theorem]{Definition}
\newtheorem{remark}[theorem]{Remark}

\numberwithin{equation}{section}

\newcommand{\R}{\mathbb{R}}

\newcommand{\e}{\varepsilon}

\renewcommand{\epsilon}{\varepsilon}

\renewcommand{\leq}{\leqslant}
\renewcommand{\le}{\leqslant}

\renewcommand{\ge}{\geqslant}

\title{On a Minkowski geometric flow in the plane:
evolution of curves with lack of scale invariance
\thanks{This work has been supported by the Australian
Research Council Discovery Project ``N.E.W. Nonlocal Equations at Work''
and it has been carried out during a very pleasant visit of the first and third
authors at the University of Pisa. We thank the Referees for their very useful comments.}}

\author{Serena Dipierro,
Matteo Novaga and
Enrico Valdinoci}

\date{} 

\begin{document}

\maketitle 

\begin{abstract}
We consider a planar geometric flow in which the normal velocity
is a nonlocal variant of the curvature. The flow is not scaling invariant
and in fact has different behaviors at different spatial scales, thus
producing phenomena that are different with respect to
both the classical mean curvature flow
and the fractional mean curvature flow.

In particular, we give examples of neckpinch singularity formation,
and we discuss convexity properties of the evolution.

We also take into account traveling waves for this geometric flow, showing that
a new family of $C^{1,1}$ and convex traveling sets
arises in this setting. 
\end{abstract}


\section{Introduction}

In this paper we consider a planar geometric
flow and we discuss its basic properties, such as
singularity formation, convexity preserving
and existence of traveling waves.
This geometric flow 
is the gradient flow of a nonlocal perimeter which is not
invariant under scaling, therefore the evolution
of a set presents different properties
at different scales (in this, the flow has natural
applications in image digitalization,
especially when tiny details have to be preserved after denoising,
as in the case of fingerprints storage).
\medskip

The mathematical framework in which we work is the following.
For any set~$E\subset\R^2$ with~$C^2$ boundary
and any~$x\in\partial E$
we denote by~$B_{r,x}^{\mathrm{ext}}$ the ball\footnote{For
consistency,
we take here balls to be open.}
of radius~$r>0$ which is locally externally
tangent to~$\partial E$ at~$x$ (that is, $B_{r,x}^{\mathrm{ext}}:=
B_r(x+r\nu_E(x))$, where~$\nu_E$ is the external unit normal).

Similarly, we denote
by~$B_{r,x}^{\mathrm{int}}$ the ball of radius~$r$ which is locally internally
tangent to~$\partial E$ at~$x$ (that is, $B_{r,x}^{\mathrm{int}}:=
B_r(x-r\nu_E(x))$).

We also denote by~${\kappa(E,x)}$ the curvature of~$\partial E$ at the point~$x$.
Then we define the $r$-curvature of~$\partial E$ at~$x$ as
\begin{equation}\label{KAPPA}
\begin{split}
&\kappa_r(E,x):=\kappa_r^+(E,x)+\kappa_r^-(E,x),\\{\mbox{where }}\quad&
\kappa_r^+(E,x):=\left\{
\begin{matrix}
\displaystyle\frac{\kappa(E,x)}2+\frac{1}{2r} & 
{\mbox{if }} B_{r,x}^{\mathrm{ext}}\subseteq\R^2\setminus E,\\
0&{\mbox{ otherwise,}}
\end{matrix}
\right.
\\{\mbox{and }}\quad&
\kappa_r^-(E,x):=\left\{
\begin{matrix}
\displaystyle\frac{\kappa(E,x)}2-\frac{1}{2r} & 
{\mbox{if }} B_{r,x}^{\mathrm{int}}\subseteq  E,\\
0&{\mbox{ otherwise,}}
\end{matrix}
\right.
\end{split}
\end{equation}
see formulas (2.10), (2.11) and~(2.12) and Lemma~2.1 in~\cite{MR3023439}.

\begin{figure}
    \centering
    \includegraphics[width=12cm]{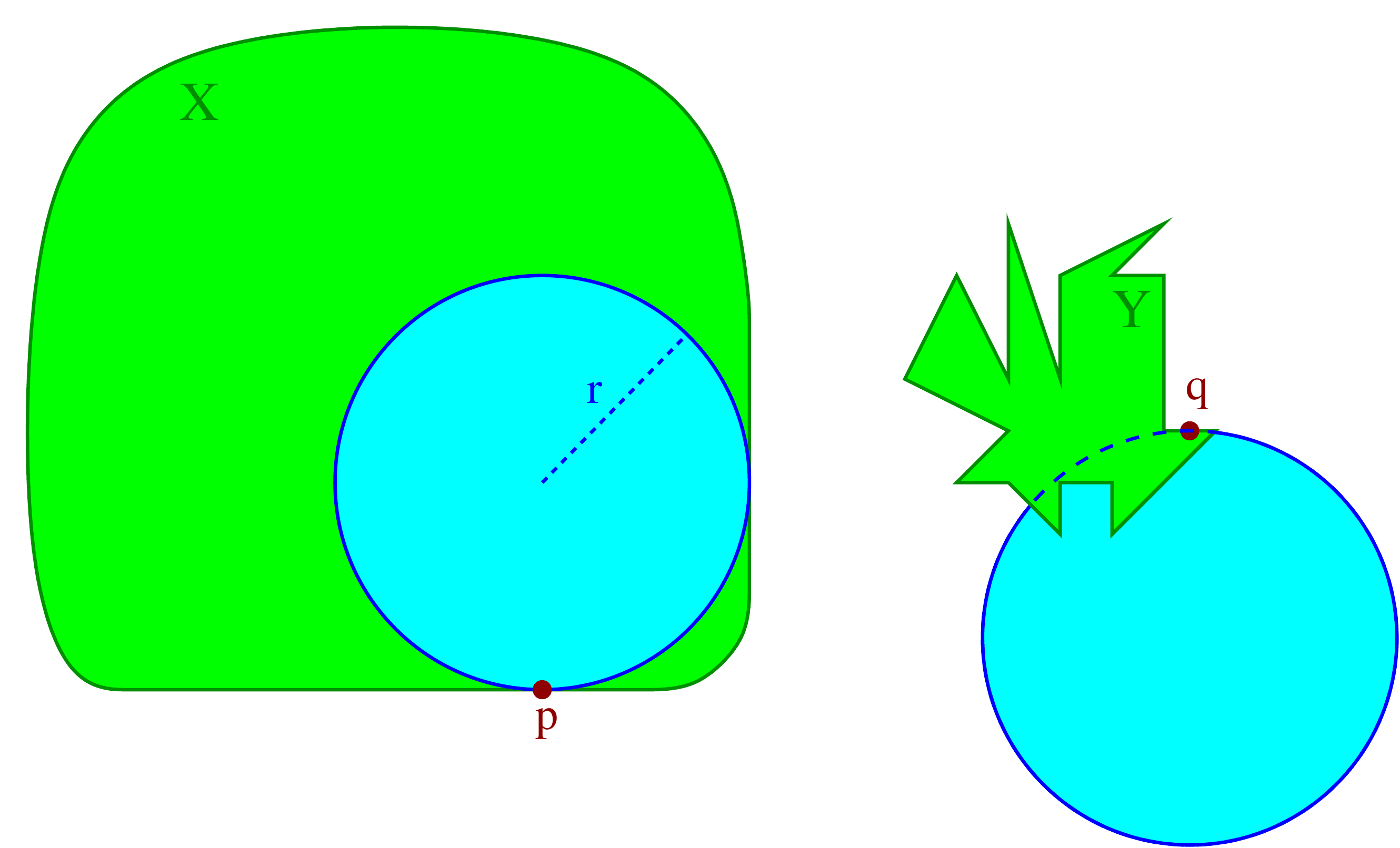}
    \caption{\em {{Examples for~$\kappa_r$.}}}
    \label{C2}
\end{figure}

We point out that,
differently from the case of the classical curvature, there are sets with~$C^2$ boundary for which~$\kappa_r$ is not continuous
(see for instance Figure~\ref{C2}, in which the set~$X$ has~$C^2$
boundary, but $\kappa_r$ jumps from~$0$ to~$\frac1{2r}$ at
the point~$p$, and the set~$Y$, which is not~$C^2$ in an $r$-neighborhood of~$q$,
but for which~$\kappa_r$ is identically zero in such a neighborhood).\medskip

Fixed~$r\in(0,1]$, we consider here the geometric flow with normal velocity equal to~$\kappa_r$,
namely if~$E_t$ denotes the evolution of a set~$E\subset\R^2$
and~$x_t\in\partial E_t$, we study the equation
\begin{equation}
\label{FLOW} \partial_t x_t\cdot\nu_{E_t}(x_t)=-\kappa_r(E_t,x_t). 
\end{equation}
The existence of a solution, in the viscosity sense, of this nonlocal geometric
problem has been established\footnote{As a technical
remark, we point out that in~\cites{MR3023439, MR3156889, MR3401008}
a smoothed version of~$\kappa_r$ (which can be considered
as a nonlocal curvature~$\kappa_f$ depending on a given function~$f$)
is taken into account,
and an existence and uniqueness result is established for this flow.
By approximating~$\chi_{(0,r)}$ with a smooth function~$f$
and taking limits, one could deduce from this the existence
of a viscosity solution for the flow driven by~$ \kappa_r$. For additional
details on this, see the forthcoming Section~\ref{AP}.}
in~\cites{MR3023439, MR3156889, MR3401008}.
The geometric equation in~\eqref{FLOW} can be seen
as the gradient flow of a nonlocal functional built by the approximated Minkowski
content, see~\cites{MR2655948, 2017arXiv170403195C}.

Of course, the quantity~$r$ in~\eqref{FLOW} plays a special
role, producing a discontinuity in the velocity field of the geometric
flow and detecting special features ``at a small scale''.
Therefore, to perform our analysis, we consider a special class
of sets, which are ``slim'' (or ``pudgy'') with respect to such a scale. 

\begin{definition}
A set~$E\subseteq\R^2$ is called ``$r$-pudgy''
if it contains a ball of radius~$r$. Otherwise, it is called ``$r$-slim''.
\end{definition}

A particular case of~$r$-slim sets is given by those which have the ``diameter
in one direction'' that is less than~$r$:

\begin{definition}
A set~$E\subseteq\R^2$ is called ``$r$-thin''
if, after a rigid motion,
it holds that~$E\subseteq \R\times(-\rho,\rho)$, with~$\rho\in(0,r)$.
\end{definition}

The first problem that we take into account is the possible
formation of neckpinch singularities in the flow defined by~\eqref{FLOW}.
We recall that in the classical mean curvature flow (or curve shortening flow),
Grayson's Theorem~\cite{MR906392} gives that no singularity
occurs
in the plane, and, in fact, the
initial set becomes convex and then shrinks smoothly towards a point.
Interestingly, this result is not true for the planar
nonlocal geometric
flow in~\eqref{FLOW} and neckpinch singularities
occur.

We construct two families of counterexamples, one for $r$-thin
and one for~$r$-pudgy sets. The first result is the following:

\begin{theorem}[Neckpinch singularity formation for $r$-thin sets]\label{TH1}
Assume that~$r$ is sufficiently small. Then,
there exists an
$r$-thin
connected set~$E\subset (-1,1)\times\left(-\frac{r}2,\frac{r}2\right)$,
with $C^\infty$ boundary
and such that any
viscosity solution of the $r$-mean curvature flow~\eqref{FLOW} starting from~$E$
does not
shrink to a point (and any viscosity evolution of~$E$ becomes disconnected).

Such~$E$ has a ``narrow dumbbell shape'' in the sense that it is obtained
by gluing two balls of radius~$r/4$ with a neck contained in~$
(-1,1)\times\left(-\frac{r}{100},\frac{r}{100}\right)$.
\end{theorem}

\begin{figure}
    \centering
    \includegraphics[width=12cm]{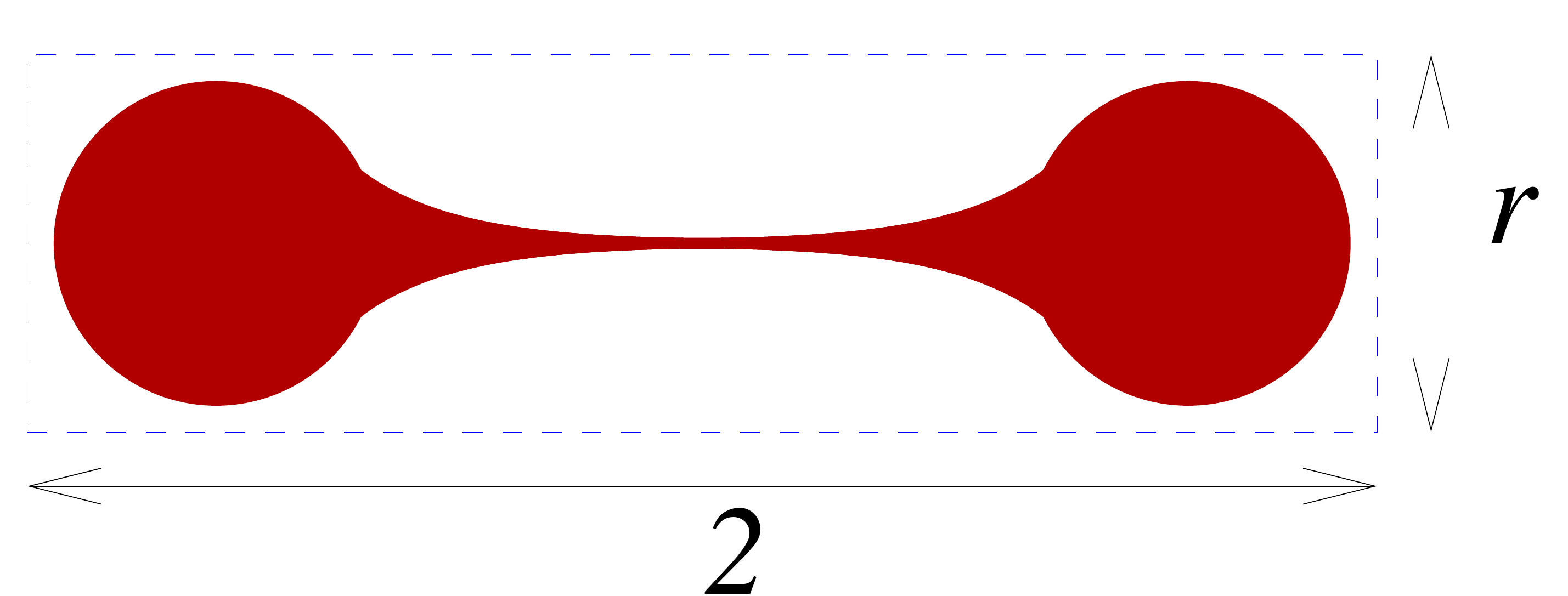}
    \caption{\em {{The set in Theorem~\ref{TH1}.}}}
    \label{WI77}
\end{figure}

The idea of the set constructed in Theorem~\ref{TH1} is depicted in Figure~\ref{WI77}.
Roughly speaking, the vertical trapping of the set
will force the inner $r$-curvatures~$\kappa^-_r$ in~\eqref{KAPPA} to vanish,
while the outer $r$-curvatures~$\kappa^+_r$ will make the neck of the set shrink
faster than the two balls on the side, thus producing the singularity.\medskip

In a sense, the example in Theorem~\ref{TH1} is quite ``pathological'' since
the singularity is produced by all the sets lying in a very small slab.
Next result provides instead an example of a dumbbell in which the
two initial balls have radius of order one, and still the evolution
produces a singularity:

\begin{theorem}[Neckpinch singularity formation for $r$-pudgy sets]\label{TH2}
Assume that~$r$ is sufficiently small. Then,
there exist~$R> r$ and
an $R$-pudgy connected set~$E\subset (-10,10)\times(-10,10)$,
with $C^\infty$ boundary
and such that any
viscosity solution of the $r$-mean curvature flow~\eqref{FLOW} starting from~$E$
does not
shrink to a point
(and any viscosity evolution of~$E$ becomes disconnected).

Such~$E$ has a ``fat dumbbell shape'' in the sense that it is obtained
by gluing two balls of radius~$R$
and with a neck contained in~$
(-10,10)\times\left(-\frac{r}{100},\frac{r}{100}\right)$.
\end{theorem}

A picture of the set~$E$ in Theorem~\ref{TH2} is sketched in Figure~\ref{WIK}
on page~\pageref{WIK}.
\medskip

We notice that the formation of neckpinch singularities also in low
dimension is a treat shared by other nonlocal geometric flows, see~\cite{2016arXiv160708032C}.
Nevertheless, the case in~\eqref{FLOW} is conceptually quite different than that
in~\cite{2016arXiv160708032C}, since the latter is scaling invariant
and the nonlocal aspect of the curvature involves the global
geometry of the set (while~\eqref{FLOW} is not scaling invariant
and the calculation of~$\kappa_r$ only involves a neighborhood of fixed side
of a given point).\medskip

Interestingly, the examples considered in Theorems~\ref{TH1} and~\ref{TH2}
have initial sets possessing a ``large curvature'' at some points.
We think that it is interesting to investigate whether or not 
singularities may also emerge when the initial set has curvatures
that are controlled uniformly when~$r$ is small.\medskip

Now, we study the convexity preservation for the flow in~\eqref{FLOW}.
This is a classical topic in the setting of the mean curvature flow,
see e.g.~\cite{MR840401}, and, once again, the behavior of the solutions of~\eqref{FLOW}
turns
out to be very different from the classical case. 
In details, the results that we have are the following:

\begin{theorem}[Convexity preserving for $r$-thin sets]\label{TH3}
Let~$E_t$ be a smooth evolution 
($C^1$ in time and~$C^2$ in space)
of a set~$E$
according to the flow in~\eqref{FLOW}. Suppose that~$E$ is $r$-thin
and convex. Then so is~$E_t$ (till the extinction time).
\end{theorem}

\begin{theorem}[Convexity loss for $r$-pudgy sets]\label{TH4}
There exists a smooth convex set~$E$ which cannot have an evolution~$E_t$
according to~\eqref{FLOW}
which is~$C^1$ in time for~$t\in[0,T)$
 and which preserves the convexity for~$t\in(0,T)$.
\end{theorem}

Roughly speaking, Theorem~\ref{TH4}
states that there exists some initial set~$E$ such that
the corresponding flow~$E_t$ cannot be at the same time regular and convex.
As for the regularity requested in Theorem~\ref{TH4},
it is assumed that, for some time~$T>0$, the boundary of the set~$E_t$
is locally parameterized by a convex function in the space variable
which is $C^1$ in time up to~$t=0$.  

\begin{figure}
    \centering
    \includegraphics[width=8cm]{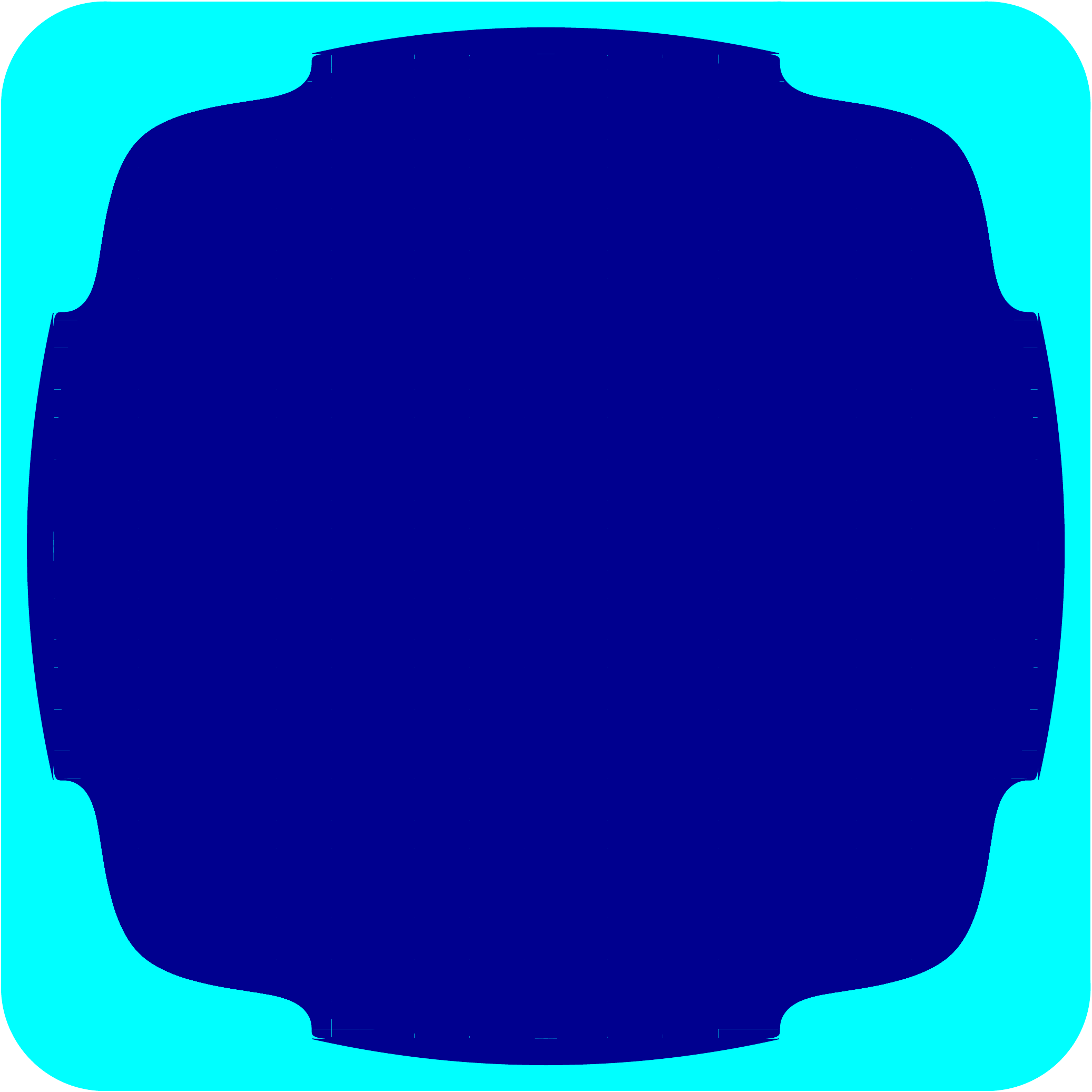}
    \caption{\em {{The possible loss of convexity in Theorem~\ref{TH4}.}}}
    \label{W823495-123I77}
\end{figure}

We have sketched in Figure~\ref{W823495-123I77}
a possible loss of convexity for the geometric flow in~\eqref{FLOW}
(a quantitative version of this picture will ground the rigorous
analysis performed in Section~\ref{PTH4}).

It is interesting to observe that Theorem~\ref{TH4} highlights an important difference with respect to
the classical curvature evolution flow in the plane, in which
if the initial curve is convex then the evolving
curves are all smooth and convex, till they shrink to a point,
becoming closer and closer to a round ball, see
page~70 in~\cite{MR840401} for detailed results
(and also~\cite{MR772132} for related higher-dimensional results
for mean curvature flows).

We refer to~\cites{2015arXiv151106944S, 2016arXiv160307239C}
for results on the preservation of a nonlocal mean curvature and
of the convex structure of a set under a fractional mean curvature evolution
(but the situation of~\eqref{FLOW} here is very different,
in light of Theorem~\ref{TH4}, 
which underlies a different behavior with respect to
the classical
mean curvature flow and also with respect to
the fractional mean curvature flow).
\medskip

It is interesting
to observe that
the example constructed in Theorem~\ref{TH4}
starts from an initial set possessing a ``large curvature'' at some points.
We think that it is interesting to investigate whether or not 
convexity is preserved if the initial set has curvatures that are
positive and bounded uniformly when~$r$ is small.\medskip

Now, we consider traveling waves for the flow in~\eqref{FLOW},
i.e. solutions of~\eqref{FLOW} in which~$E_t$ is of the form
\begin{equation}\label{TWW} y> h(x)+ct,\end{equation}
for some real function~$h$, and~$c\in\R$. In this setting,
we have that the geometric flow in~\eqref{FLOW}
presents a new class of traveling waves, which are obtained
by gluing together a convex function depending
on~$r$ near the origin and 
the ``standard grim reaper'' at infinity:

\begin{theorem}\label{TW} For any~$c\in (0,r)$
there exists a traveling wave for the geometric flow
in~\eqref{FLOW} with speed equal to~$c$.
The corresponding traveling set is~$C^{1,1}$ and convex.
\end{theorem}

A more precise description of the shape of this traveling wave will be given
in the forthcoming formula~\eqref{hstar}. See also Figure~\ref{NEWG}
for a picture of this traveling wave.
\medskip

The rest of the paper is organized as follows. In Section~\ref{AP},
we first discuss some ``pathologies'' of the geometric flow in~\eqref{FLOW}
and recall an approximation scheme exploited in~\cites{MR2728706, MR3023439, MR3401008}
(such an approximation is not explicitly used here, but it provides
a conceptual framework for the flow in~\eqref{FLOW} from a viscosity perspective).
In Section~\ref{PAO1} we consider the neckpinch formation for $r$-thin sets
and we prove Theorem~\ref{TH1}. Then, in
Section~\ref{PAO2} we consider the neckpinch formation for $r$-pudgy sets
and we prove Theorem~\ref{TH2}. The convexity preservation for $r$-thin sets
is discussed in Section~\ref{PAO3}, where we present the proof
of Theorem~\ref{TH3}. The possible
loss of convexity and the proof of Theorem~\ref{TH4}
are presented in Section~\ref{PTH4}, and the traveling waves, with the proof
of Theorem~\ref{TW}, are discussed in Section~\ref{TWA}.

\section{A viscosity approximation of~\eqref{FLOW}}\label{AP}

The geometric flow in~\eqref{FLOW} is rather special,
given its lack of invariance and different behaviors at different scales.
Also, the velocity field is discontinuous (even for convex sets)
at points where tangent balls possess two or more projections along the boundary.
This lack of regularity in the velocity produces
some instability properties in the set evolution of~\eqref{FLOW} (corresponding to a ``fattening''
of the associated evolution by level sets).
For instance, if one considers the initial set~$E:=\R\times(-\ell,\ell)$,
from~\eqref{KAPPA} it holds that~$\kappa_r=0$ if~$\ell>r$
and~$\kappa_r=\frac1{2r}$ if~$\ell\in(0,r)$. Therefore this set stays put
under the geometric flow in~\eqref{FLOW} if~$\ell>r$,
but it shrinks to a line in finite time if~$\ell\in(0,r)$.\medskip

Due to phenomena of this sort, to compensate
the lack of continuity of the velocity field in~\eqref{FLOW},
it is desirable to approximate
this flow with a more regular one.
For this, we recall a procedure discussed in
Section~6.4 of~\cite{MR3401008}.
We consider a smooth function~$
f : \R \to[0,1]$, which is
even, supported in~$[-r,r]$ and such that~$f(x)=1$ for any~$x\in\left[-\frac{r}2,\frac{r}2\right]$,
and~$f'(x)\le0$ for any~$x\ge0$.
Recalling~\eqref{KAPPA},
for any~$x\in\partial E$ one defines
$$ \kappa_f(x,E):=\kappa^+_f (x,E)
+\kappa^-_f(x,E),$$ where 
\begin{equation}\label{93994-129348-23948}
\begin{split}&
\kappa^+_f (x, E) := -\int_0^r \frac\sigma{r} \,f'(\sigma)\, \kappa^+_\sigma (x, E) \,d\sigma\\
{\mbox{and }}\;&
\kappa^-_f (x, E) := -\int_0^r \frac\sigma{r} \,f'(\sigma)\,
\kappa^-_\sigma (x, E) \,d\sigma.\end{split}
\end{equation}
Then, one can consider the geometric flow associated to~$\kappa_f$,
that is, the flow in~\eqref{FLOW} with~$\kappa_r$ replaced by~$\kappa_f$,
\begin{equation}
\label{FLOWf} \partial_t x_t\cdot\nu_{E_t}(x_t)=-\kappa_f(E_t,x_t). 
\end{equation}
This flow can be seen as an approximation of that in~\eqref{FLOW}
and it is used in~\cite{MR3023439} to establish uniqueness results
and in~\cite{MR2728706} for numerical purposes.\medskip

In this article, we will not make use explicitly of the flow
in~\eqref{FLOWf} (though similar arguments as the ones exploited here
may be used in this framework as well), but we expect that
solutions of~\eqref{FLOW} emerge from an appropriate limit
of the viscosity
solutions of~\eqref{FLOWf} as the function~$f$ approaches
the characteristic function of~$(0,r)$. To rigorously perform
such a limit procedure, one has to check uniform continuity
of the viscosity solutions of~\eqref{FLOWf}.\medskip

Though several notions of solutions are possible for the geometric flow under consideration,
for the sake of concreteness we consider here the following one,
inspired by a viscosity approach introduced in~\cite{MR1216585},
see also~\cite{MR1467354}.
We now define~${\mathcal{S}}_{{\mbox{fast}}}$ as
the collection of sets~$\mathcal{U}_t\subset\R^2$
whose evolution is~$C^{1}$ in time and~$C^2$ in space,
and with velocity strictly higher than the one prescribed by the geometric flow~\eqref{FLOW}.
Namely, given~$b>a\ge 0$, we consider a map~$[a,b]\ni t\mapsto\mathcal{U}_t\subset\R^2$
and we assume that:
\begin{itemize}
\item there exists a bounded open set $A\subset\R^2$ 
such that $\partial\mathcal{U}_t\subset A$ for all $t\in[a,b]$,
\item the signed distance function~$\R^2\times [0,T]\ni(x,t)\mapsto
{{\mbox{dist}}}_{\mathcal{U}_t}(x)$ from the boundary of~$\mathcal{U}_t$ is $C^1$ in~$t\in(a,b)$,
$C^0$ in~$t\in[a,b]$
and~$C^2$ in $x\in A$,
\item if~$v(x,t)$ denotes the inner normal velocity of~$\partial\mathcal{U}_t$ at~$x\in\partial\mathcal{U}_t$,
then~$v(x,t)>\kappa_r(\mathcal{U}_t,x)$.
\end{itemize}
Then, we let ${\mathcal{S}}_{{\mbox{fast}}}$ be the collection of all~$\{\mathcal{U}_t\}_{t\in [a,b]}$
with such properties.

Similarly, we define~${\mathcal{S}}_{{\mbox{slow}}}$ as
the collection of sets~$\mathcal{V}_t$
which possess an evolution that is~$C^{1}$ in time and~$C^2$ in space,
and with velocity strictly smaller than that prescribed by the geometric flow.

In this setting, we say that~$\{E_t\}_{t\in [0,T)}$ is a evolution of the geometric flow under consideration
if for any~$\{\mathcal{U}_t\}_{t\in
[a,b]}\in {\mathcal{S}}_{{\mbox{fast}}}$ and any~$\{\mathcal{V}_t\}_{t\in[a,b]}\in {\mathcal{S}}_{{\mbox{slow}}}$
such that $0\le a<b<T$
and~${\mathcal{U}_0}\subset E_0\subset{\mathcal{V}}_0$, it holds that~${\mathcal{U}_t}\subset E_t\subset{\mathcal{V}}_t$
for all~$t\in[0,T)$. Existence of such evolutions for very general flows 
which include~\eqref{FLOW} can be easily proved by Perron's method (see~\cite{MR1467354}).
In~\cite{MR1216585} it is proved that, in the case of the mean curvature flow,
any generalized solutions of this kind is contained in the zero-level set
of the viscosity solutions.

Notice that comparison principles with evolutions in ${\mathcal{S}}_{{\mbox{fast}}}$ and ${\mathcal{S}}_{{\mbox{slow}}}$
are automatic in this setting. We do not address here
the problem of constructing these types of solutions for the geometric flow under
consideration and we do not address the study of the uniqueness properties of solutions
in this class. 
\section{Proof of Theorem \ref{TH1}}\label{PAO1}

\subsection{Geometric barriers}

We start with the construction of a narrow barrier. To this aim,
we fix~$\eta\in\left(0,\frac{r}{64\pi^2}\right)$ and we define
$$ G_{\eta}:=\Big\{
(x,y)\in\R^2 {\mbox{ s.t. }} |y|\le \eta+\frac{r}{32\pi^2}\big(1-\cos(4\pi x)\big)
\Big\} .$$
Then, we have:

\begin{lemma}\label{0oerfetuu}
If~$r$ is sufficiently small
then
\begin{equation}\label{0pdk-12pweXdlf}
\kappa_r(G_{\eta},p)\ge \frac{1}{4r}\qquad{\mbox{for any }}p=(p_1,p_2)\in\partial
G_{\eta} .\end{equation}
\end{lemma}

\begin{proof} Let
$$g_\eta(x):=\eta+\frac{r}{32\pi^2}\big(1-
\cos(4\pi x)\big).$$
We have that
\begin{eqnarray*}
&& |g_\eta'(x)|=\left|\frac{r}{8\pi} \sin(4\pi x)\right|\le \frac{r}{8\pi} \\{\mbox{and }}
&& |g_\eta''(x)|=\left|\frac{r}{2} \cos(4\pi x)\right|\le \frac{r}{2}.
\end{eqnarray*}
This implies that the curvature of the graph of~$g_\eta$
is bounded in absolute value by~$\frac{r}{2}$, provided that~$r$ is sufficiently
small, and therefore~$G_\eta$ can always be touched from outside
by a ball of radius~$r$. Consequently, by~\eqref{KAPPA}, it holds that
\begin{equation}\label{01-23osd}
\kappa_r^+(G_\eta,x)=\frac{\kappa(G_\eta,x)}2+\frac{1}{2r}\ge
-\frac{r}{2}+\frac{1}{2r}.
\end{equation}
On the other hand, the vertical diameter of~$G_\eta$ is~$\frac{r}{16\pi^2}$
and so
no ball of radius~$r$ can be contained inside~$G_\eta$. Therefore, by~\eqref{KAPPA},
we conclude that~$\kappa_r^-(G_\eta,x)=0$.
This and~\eqref{01-23osd} give that
$$
\kappa_r(G_\eta,x)\ge
-\frac{r}{2}+\frac{1}{2r},$$ from which the desired result follows.
\end{proof}

\begin{figure}
    \centering
    \includegraphics[width=13cm]{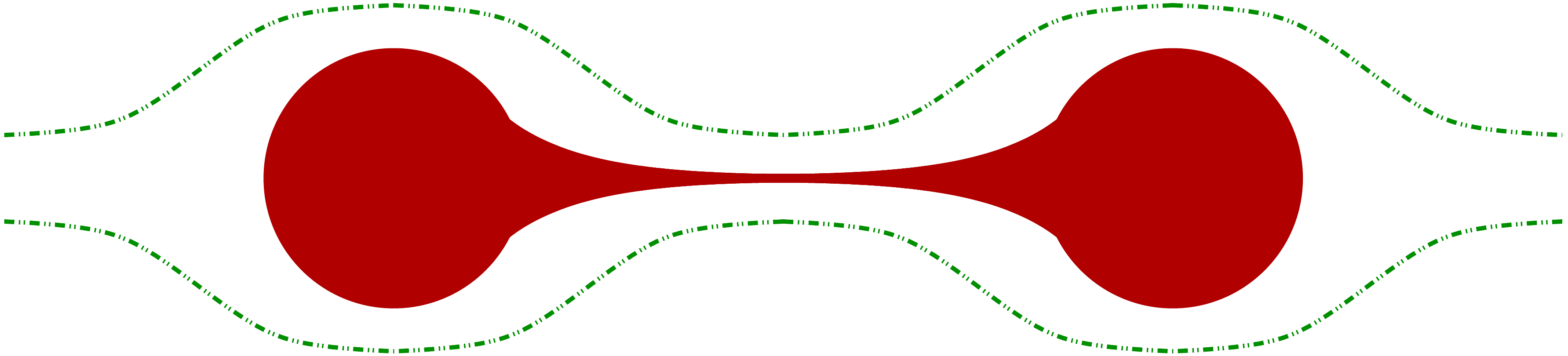}
    \caption{\em {{The sets~$E$ and~$G_{\eta_0}$ in the proof of Theorem~\ref{TH1}.}}}
    \label{923ccv40WI77}
\end{figure}

\subsection{Completion of the proof of Theorem \ref{TH1}}\label{SECla21}

With Lemma~\ref{0oerfetuu}, we can complete the proof
of
Theorem \ref{TH1}. For this, we take~$t_r>0$ to
be the extinction time of
the ball~$B_{r/10^6}$. We define
$$ \eta_0:=\min\left\{\frac{t_r}{8r},\,\frac{r}{128\pi^2}\right\}\quad
{\mbox{ and }}\quad\eta(t):=\eta_0-\frac{t}{4r}.$$
We take $q_\pm:=\left(\pm\frac12,0\right)$
and
$$ N:=
[-q_-,q_+]\times \left[ -\frac{\eta_0}2,\frac{\eta_0}2 \right].$$ We consider
a connected and smooth set~$E\subset(-1,1)\times\left(-\frac{r}2,\frac{r}2\right)$
such that
$$ G_{\eta_0}\supseteq E\supseteq B_{r/10^6}(q_-)\cup B_{r/10^6}(q_+)\cup N,$$
with~$E\cap\{ |x|\le 1/10\}=N$,
see Figure~\ref{923ccv40WI77}. Confronting with halfplanes, the comparison principle
for~\eqref{FLOW} (see e.g. Section~3.3
in~\cite{MR3023439})
gives that also the evolution~$E_t$ lies in~$
(-1,1)\times\left(-\frac{r}2,\frac{r}2\right)$.
In addition, the velocity of~$G_{\eta(t)}$ in modulus
coincides with~$\frac{1}{4r}$, which is less than
the $r$-curvature of~$G_{\eta(t)}$, thanks to~\eqref{0pdk-12pweXdlf}.
As a consequence of this and of the comparison principle, we obtain that~$E_t
\subseteq G_{\eta(t)}$. Since $t_\star:=4r\eta_0<t_r$
and~$\eta(t_\star)=0$, we have that~$
G_{\eta(t_\star)}$ develops a neck singularity, and so does~$E_t$ for some~$t\in(0,t_\star)$.
This completes the proof of Theorem~\ref{TH1}.

\section{Proof of Theorem \ref{TH2}}\label{PAO2}

\subsection{Geometric barriers}

This section is devoted to the construction of
an explicit barrier for the geometric flow in~\eqref{FLOW}. Roughly speaking,
this barrier is constructed by taking the region trapped between a graph
and its reflection along the horizontal axis. Such graph is constructed
by interpolating a parabola with curvature comparable to~$r$ near the origin
with a uniformly concave function. The interpolation will occur when
the values on the abscissa are of order~$r$ and the functions are also of order~$r$,
but the gradients are of order~$1$. This quantitative construction
is needed to compute efficiently the $r$-curvatures in~\eqref{KAPPA}
and the example that we provide may turn out to be useful also in other cases.

We fix~$M\ge1$, to be taken appropriately large in the sequel.
We also consider a bump function~$\varphi\in C^\infty_0([-2,2],\,[0,1])$,
with~$\varphi=1$ in~$[-1,1]$, $|\varphi'|\le2$
and~$|\varphi''|\le2$.
We also define~$\rho:=Mr$ and, for any~$x\in\R$,
$$ g(x):=\frac{x^2}{2M^2\rho}\,\varphi\left(\frac{x}{\rho}\right)+
\left( 1-\varphi\left(\frac{x}{\rho}\right)\right)\,\frac{|x|}{M^2\,(1+|x|)}.$$
For any~$\e>0$ we set~$g_\e(x):=\e+g(x)$ and
$$ F_{\e}:=\Big\{
(x,y)\in\R^2 {\mbox{ s.t. }} |y|\le g_\e(x)
\Big\} .$$

\begin{figure}
    \centering
    \includegraphics[width=12cm]{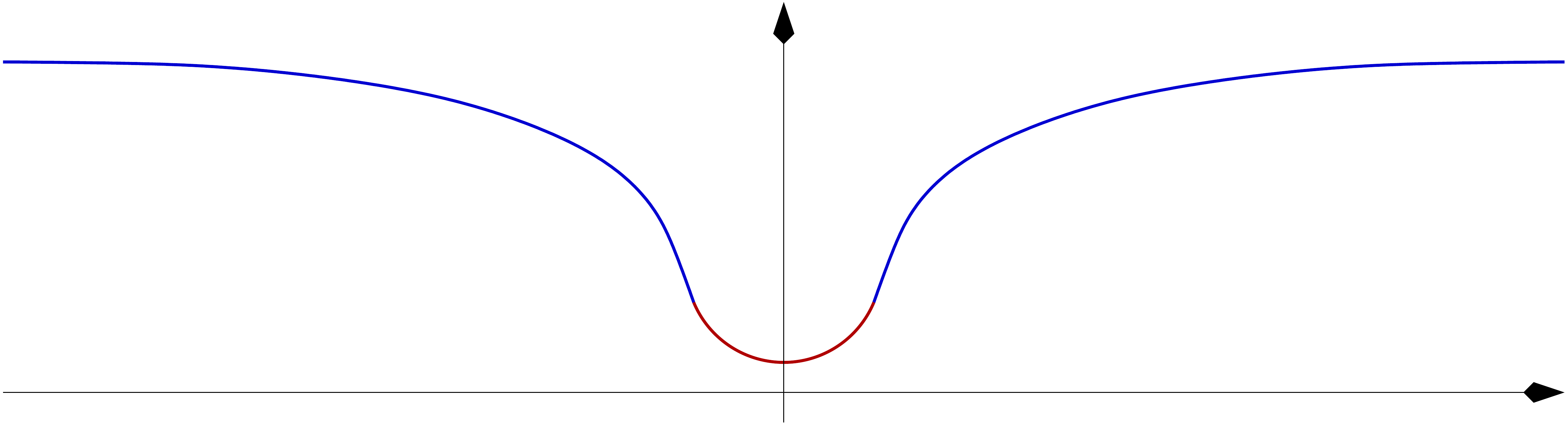}
    \caption{\em {{The function~$g_\e$.}}}
    \label{WIK2}
\end{figure}

The graph of~$g_\e$ is depicted in Figure~\ref{WIK2}.
The set~$F_{\e}$ is a useful barrier for the $r$-geometric
flow, according to the following calculation:

\begin{lemma}\label{LEMMA BAR92w0:01}
There exist~$M>1$ and~$c_0\in(0,1)$ such that if~$r\in\left(0,\frac1M\right)$
and~$\e\in \left(0,\frac{r}{M}\right)$ then
\begin{equation}\label{1}
\kappa_r(F_{\e},p)\ge c_0\qquad{\mbox{for any }}p=(p_1,p_2)\in\partial
F_{\e} {\mbox{ with }}|p_1|\le10.\end{equation}
\end{lemma}

\begin{proof} By symmetry, we can reduce our analysis
to the first quadrant,
i.e. prove~\eqref{1}
for~$p=(p_1,p_2)\in\partial
F_{\e}$, with~$p_1\in[0,10]$ and~$p_2\in[0,+\infty)$.
For any~$x\in [0,10]$, it holds that
\begin{eqnarray*}
g'_\e(x)&=&\frac{x}{M^2\rho}\,\varphi\left(\frac{x}{\rho}\right)+
\frac{x^2}{2M^2\rho^2}\,\varphi'\left(\frac{x}{\rho}\right)\\&&\qquad
-\varphi'\left(\frac{x}{\rho}\right)\,\frac{x}{M^2\,\rho\,(1+x)} 
+\left( 1-\varphi\left(\frac{x}{\rho}\right)\right)\,\frac{1}{M^2\,(1+x)^2}\end{eqnarray*}
and
\begin{eqnarray*}
g''_\e(x)&=
&\frac{1}{M^2\rho}\,\varphi\left(\frac{x}{\rho}\right)+
\frac{2x}{M^2\rho^2}\,\varphi'\left(\frac{x}{\rho}\right)+
\frac{x^2}{2M^2\rho^3}\,\varphi''\left(\frac{x}{\rho}\right)\\&&\qquad
-\varphi''\left(\frac{x}{\rho}\right)\,\frac{x}{M^2\,\rho^2\,(1+x)} 
-\varphi'\left(\frac{x}{\rho}\right)\,\frac{2}{M^2\,\rho\,(1+x)^2} \\&&\qquad
-\left( 1-\varphi\left(\frac{x}{\rho}\right)\right)\,\frac{2}{M^2\,(1+x)^3}.
\end{eqnarray*}
Consequently, for any~$x\in[0,10]$,
\begin{equation}\label{CAG-1}
\begin{split}
|g''_\e(x)|\,&\le
\frac{1}{M^2\rho}+
\frac{8\rho}{M^2\rho^2}+
\frac{8\rho^2}{2M^2\rho^3}
+\frac{4\rho}{M^2\,\rho^2} +
\frac{4}{M^2\,\rho}+\frac{2}{M^2}
\\ &=\frac{(21+2\rho)}{M^2\,\rho}
\\ &=\frac{(21+2Mr)}{M^3\,r}
\\ &\leq\frac{23}{M^3\,r}
.\end{split}
\end{equation}
Also,
\begin{equation}\label{9e233302}
{\mbox{for any }}x\in[2\rho,10],\qquad
|g'_\e(x)| =
\frac{1}{M^2\,(1+x)^2}\le
\frac{1}{M^2}.
\end{equation}
Moreover,
\begin{equation}\label{0qeid0e8t}
{\mbox{for any }}x\in[2\rho,10],\qquad
g''_\e(x) =
-\frac{2}{M^2\,(1+x)^3} \in \left[ 
-\frac{2}{M^2},-\frac{2}{M^2\,11^3}\right].
\end{equation}
On the other hand,
\begin{equation}\label{0w2e}
\begin{split}
{\mbox{for any }}x\in[0,2\rho],\qquad
g_\e(x)\,&\le\e+
\frac{4\rho^2}{2M^2\rho}+
\frac{2\rho}{M^2}
\\ &=\e+\frac{4\rho}{M^2}\\&\le\frac{5r}{M}.
\end{split}
\end{equation}
In the same way, we see that
\begin{equation}\label{0-2e0304}
\begin{split}
{\mbox{for any }}x\in[0,2\rho],\qquad
|g'_\e(x)|\,&\le
\frac{2\rho}{M^2\rho}+
\frac{8\rho^2}{2M^2\rho^2}
+\frac{4\rho}{M^2\,\rho} 
+\frac{1}{M^2}\\&\le\frac{11}{M^2}.
\end{split}
\end{equation}
Now, since the curvature of the graph is given by
$$ -\left( \frac{g_\e'}{\sqrt{1+(g_\e')^2}}\right)'=-
\frac{g_\e''}{(1+(g_\e')^2)^{3/2}},$$
it follows from~\eqref{CAG-1} that 
\begin{equation}\label{0-edfru5755757}
{\mbox{the curvature of the graph is
bounded everywhere in absolute value by }}
\frac{23}{M^3\,r}
,\end{equation}
which is less than~$\frac1r$ if~$M$ is sufficiently large.
Hence, the set~$F_\e$ can always be touched from outside by balls
of radius~$r$, i.e.~$B_{r,p}^{\mathrm{ext}}\subseteq\R^2\setminus F_\e$
and so, by~\eqref{KAPPA},
\begin{equation}\label{lao2}
\kappa_r^+(F_\e,p)=\frac{\kappa(F_\e,p)}2+\frac{1}{2r}.\end{equation}
Similarly, from~\eqref{9e233302} and~\eqref{0qeid0e8t}, we infer that
\begin{equation}\label{80483489284580-edfru5755757}\begin{split}&
{\mbox{the curvature of the graph with abscissa in $[2\rho,10]$}}\\&{\mbox{is
bounded from below by }}\frac{1}{M^2\,11^3}
.\end{split}\end{equation}
Now, to compute~$\kappa_r^-(F_\e,p)$ we distinguish
the two cases~$p_1\in[0,2\rho]$ and~$p_1\in[2\rho,10]$.
If~$p_1\in[0,2\rho]$ we claim that
\begin{equation}\label{FUO}
B_{r,p}^{\mathrm{int}} \cap \{ y<-g_\e(x)\}\ne\emptyset.
\end{equation}
\begin{figure}
    \centering
    \includegraphics[width=14cm]{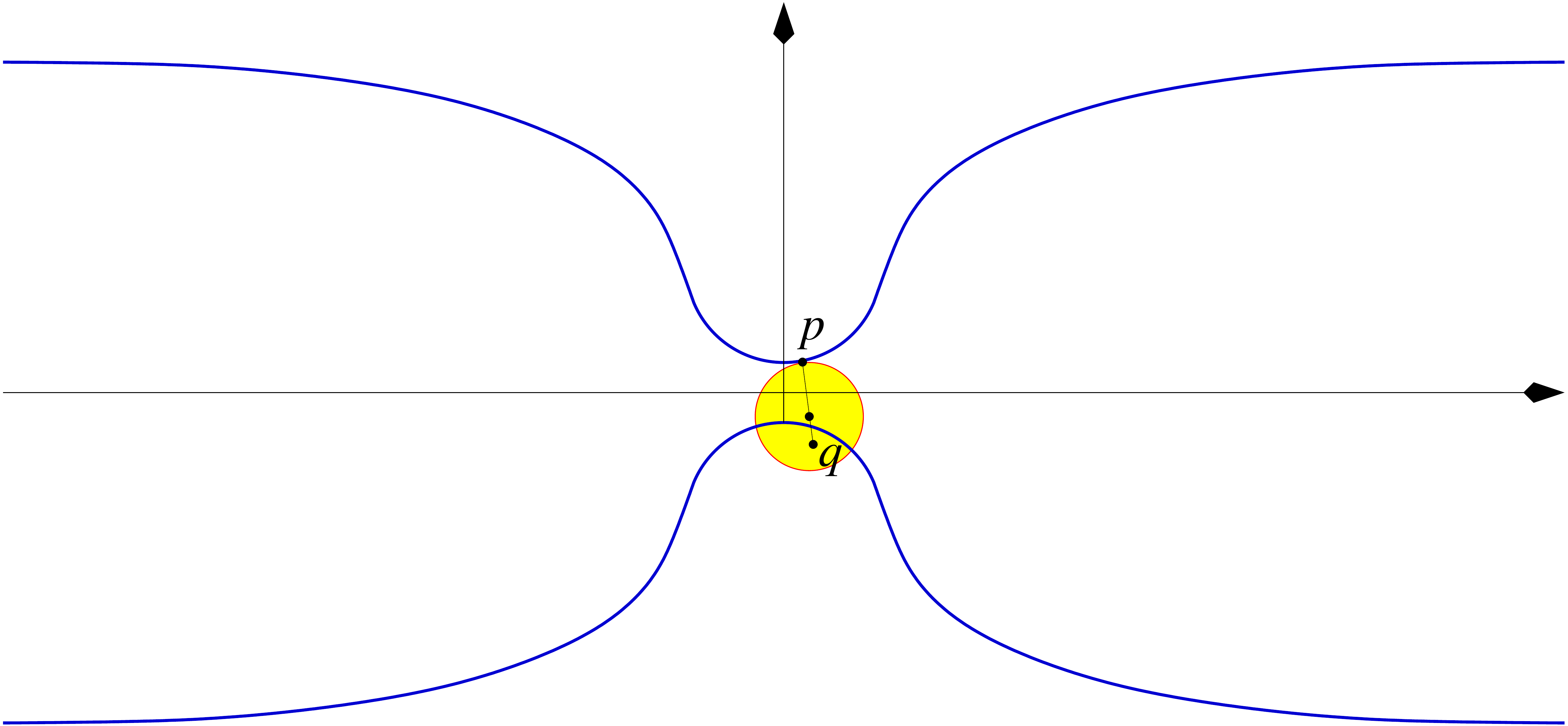}
    \caption{\em {{The geometry involved in the proof of~\eqref{FUO}.}}}
    \label{129345WIK2}
\end{figure}
To check this, we use Figure~\ref{129345WIK2}
and we notice that the exterior normal at~$p$ is given by
$$\nu_{F_\e}(p)=
\frac{(-g_\e'(p_1),1)}{ \sqrt{1+(g_\e'(p_1))^2}}$$ and
$$ q=(q_1,q_2):=p-r\nu_{F_\e}(p)-\frac{r\nu_{F_\e}(p)}{2}\in B_{r,p}^{\mathrm{int}}.$$
So, to prove~\eqref{FUO},
it is enough to check that
\begin{equation}\label{FUO2}
q_2<-g_\e(q_1).
\end{equation}
As a matter of fact, since~$p_2=g_\e(p_1)$,
\begin{eqnarray*}
q_2+g_\e(q_1) &=&
\left(p-\frac{3r\nu_{F_\e}(p)}{2}\right)_2+g_\e
\left( 
\left(p-\frac{3r\nu_{F_\e}(p)}{2}\right)_1\right)\\
&=&
g_\e(p_1)-\frac{3r}{2 \sqrt{1+(g_\e'(p_1))^2}}
+g_\e
\left(p_1+\frac{3r\,g_\e'(p_1)}{2\sqrt{1+(g_\e'(p_1))^2}}\right)\\
{\mbox{by \eqref{0-2e0304} }}
\qquad&\le& 2g_\e(p_1)
-\frac{3r}{2 \sqrt{1+(11/M^2)^2}}+\sup_{x\in[0,10]} |g'_\e(x)|\,
\frac{3r\,|g_\e'(p_1)|}{2\sqrt{1+(g_\e'(p_1))^2}}\\
{\mbox{by \eqref{9e233302}, \eqref{0w2e} and \eqref{0-2e0304} }}
\qquad&\le& \frac{10\,r}M
-\frac{3r}{2 \sqrt{1+(11/M^2)^2}}
+\left(\frac{11}{M^2}\right)^2
\,\frac{3r}{2}
\\ &\le& \frac{r}2-r,
\end{eqnarray*}
as long as~$M$ is sufficiently large.
This proves~\eqref{FUO2}, and so~\eqref{FUO}.

Then, from~\eqref{KAPPA} and~\eqref{FUO}, we obtain that
$$ {\mbox{if $p_2\in[0,2\rho]$, then $\kappa^-_r(F_\e,p)=0$.}}$$
{F}rom this and~\eqref{lao2}, we conclude that
\begin{equation*}
{\mbox{if $p_2\in[0,2\rho]$, then }}
\kappa_r(F_\e,p)=\frac{\kappa(F_\e,p)}2+\frac{1}{2r}.\end{equation*}
This and~\eqref{0-edfru5755757} imply that
\begin{equation}\label{23402-1-3232-4}
{\mbox{if $p_2\in[0,2\rho]$, then }}
\kappa_r(F_\e,p)\ge
-\frac{23}{2\,M^3\,r}
+\frac{1}{2r}\ge\frac{1}{3r},\end{equation}
as long as~$M$ is large enough.

On the other hand, by~\eqref{80483489284580-edfru5755757},
we know that
$$ {\mbox{if $p_2\in[2\rho,10]$, then }}\kappa(F_\e,p)\ge
\frac{1}{M^2\,11^3}$$
and therefore, recalling~\eqref{KAPPA} and~\eqref{lao2},
we have that
\begin{eqnarray*}
{\mbox{if $p_2\in[2\rho,10]$, then }}\kappa_r(F_\e,p)&=&
\kappa_r^+(F_\e,p)+\kappa_r^-(F_\e,p) \\&\ge&
\frac{\kappa(F_\e,p)}2+\frac{1}{2r} +\min\left\{
0,\;\frac{\kappa(F_\e,p)}2-\frac{1}{2r}
\right\}\\
&=&
\min\left\{\frac{\kappa(F_\e,p)}2+\frac{1}{2r},\;
\kappa(F_\e,p)
\right\}
\\ &\ge& \frac{1}{2\,M^2\,11^3}.\end{eqnarray*}
The desired result now follows plainly from this inequality and~\eqref{23402-1-3232-4}.
\end{proof}

\subsection{Completion of the proof of Theorem~\ref{TH2}}

With the construction in Lemma~\ref{LEMMA BAR92w0:01},
the proof of Theorem~\ref{TH2} follows by a comparison principle, with
an argument similar to that in Section~\ref{SECla21}.
We give the full argument for the facility of the reader.
We take~$M$ and~$c_0$ as in Lemma~\ref{LEMMA BAR92w0:01}
and we define, for any~$t\ge0$,
$$ \e(t):=\frac{r}{2M}-\frac{c_0\,t}{2}.$$
The set~$F_{\e(t)}$ falls under the assumption of Lemma~\ref{LEMMA BAR92w0:01},
so, by~\eqref{1},
\begin{equation}\label{78} \kappa_r(F_{\e(t)},p)\ge c_0\qquad{\mbox{for any }}p=(p_1,p_2)\in\partial
F_{\e(t)} {\mbox{ with }}|p_1|\le10.\end{equation}
Moreover, using balls
of radius~$r/2$ centered at points~$p=(p_1,0)$ with~$|p_1|$ large, we see
that the portion of the barrier~$F_{\e(t)}$ coming from infinity
does not collapse instantaneously.

Now we set~$q_\pm:=(\pm 3,0)$ and
$$ N:= [-3,3]\times \left[ -\frac{\e(0)}2,\frac{\e(0)}2 \right],$$ and
we take~$R>0$ and a connected and smooth set~$E\subset F_{\e(0)}$
such that
$$ E\supseteq B_{R}(q_-)\cup B_R(q_+)\cup N.$$

\begin{figure}
    \centering
    \includegraphics[width=12cm]{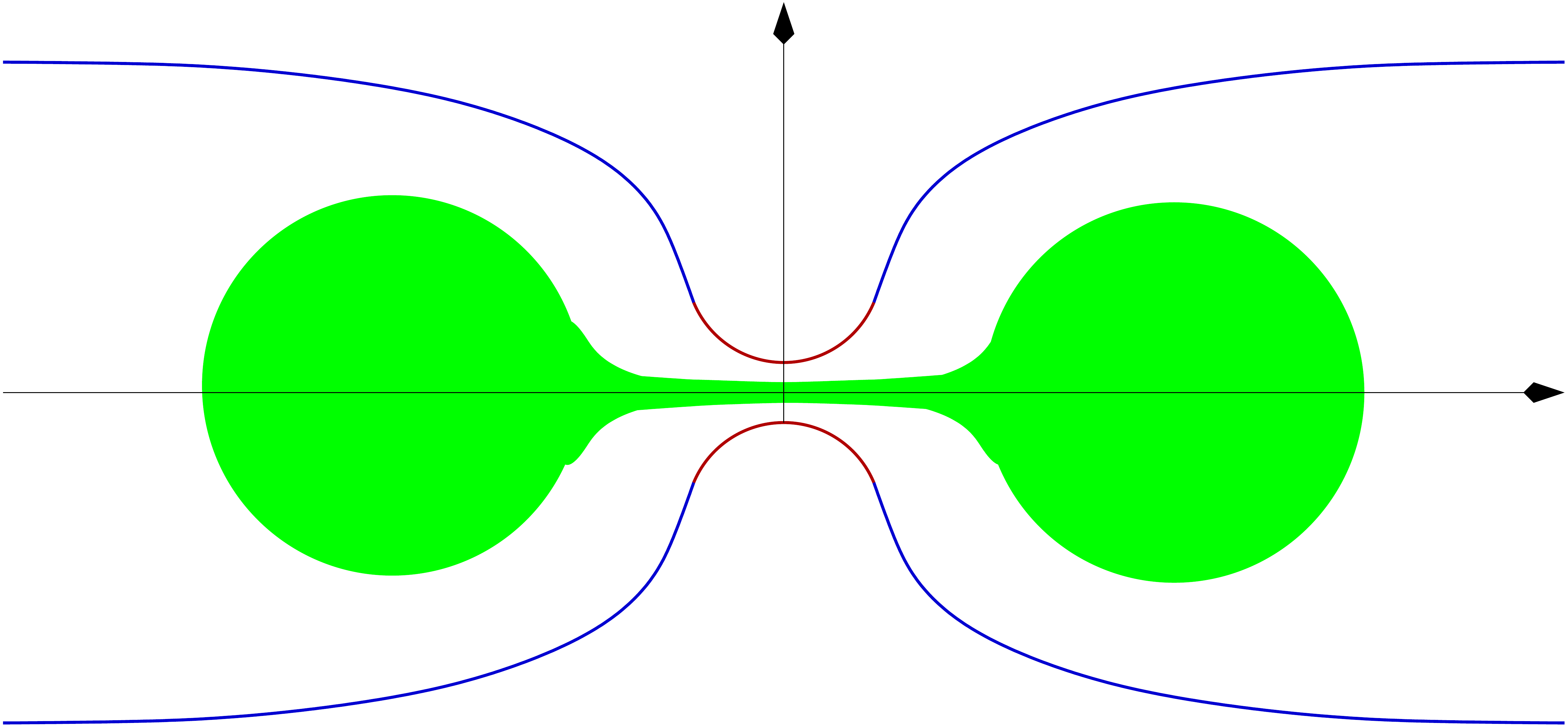}
    \caption{\em {{The sets~$E$ and~$F_{\e(0)}$.}}}
    \label{WIK}
\end{figure}

The geometric situation of this proof is depicted in
Figure~\ref{WIK}.
Notice that we can take such~$R$ independent of~$r$
and
the modulus of the velocity of~$F_{\e(t)}$
is~$\frac{c_0}{2}$, which is less than the normal velocity of the flow~\eqref{FLOW},
thanks to~\eqref{78}. Therefore, by comparison principle (see e.g. Section~3.3
in~\cite{MR3023439}),
we conclude that~$E_t\subset\{ |x|\le10\}$ and, more importantly,
$E_t\subset F_{\e(t)}$ for all~$t\in[0,T)$, being~$T$ the extinction
time. Since the extinction time~$T$ is bounded from below by the one of the ball~$B_R$,
which is independent of~$r$, we can take~$r$ sufficiently small
and suppose that~$\frac{r}{c_0\,M}<T$.
But, since~$F_{\e(t)}$ develops a neckpinch
at time~$t=\frac{r}{c_0\,M}$, also any viscosity evolution of the
set~$E_t$ develops a singularity before this time, and gets disconnected.
This completes the proof of Theorem~\ref{TH2}.

\section{Proof of Theorem \ref{TH3}}\label{PAO3}

We compute the evolution of the curvature of a geometric flow with normal
velocity~$v$, assuming that
such evolution is $C^1$ in time and $C^2$ in space.
For this, we denote by~$s$ the arclength variable
and we recall (see e.g. formula~(9) in~\cite{MR2863468})
that if a set~$F_t$ is a solution of~$\partial_t x_t\cdot\nu_{F_t}(x_t)=-v(x_t)$
then
\begin{equation}\label{902e3reru}
\partial_t \kappa(F_t,x_t)=\partial_{ss}^2 v(x_t)+\kappa^2(F_t,x_t)\,v(x_t).
\end{equation}
Furthermore, comparing with halfplanes, we see that the evolution of~$E_t$
is $r$-thin for any time~$t$ (till extinction). 
Therefore, $E_t$ does not contain balls of radius~$r$
and then, in view of~\eqref{KAPPA}, it holds that
\begin{equation}\label{0193e4-12349}
\kappa_r^-(E_t,x_t)=0\quad{\mbox{ for any }}x_t\in\partial E_t.
\end{equation}
Now, suppose that~$E_t$ is convex: it follows from~\eqref{KAPPA}
that
$$ \kappa_r^+(E_t,x_t)=\frac{\kappa(E_t,x_t)}2+\frac{1}{2r}.$$
This and~\eqref{0193e4-12349} give
that
$$ \kappa_r(E_t,x_t)=\frac{\kappa(E_t,x_t)}2+\frac{1}{2r}.$$
As a consequence,
\begin{eqnarray*}&&
\kappa_r(E_t,x_t)\ge\frac{1}{2r}\\{\mbox{and }}&&\partial^2_{ss}
\kappa_r(E_t,x_t)=\frac{\partial^2_{ss}\kappa(E_t,x_t)}2.
\end{eqnarray*}
In particular, by~\eqref{902e3reru},
we have that, if~$E_t$ is a solution of the geometric
flow in~\eqref{FLOW}, till it is convex it holds that
\begin{equation*}
\begin{split}
\partial_t \kappa(E_t,x_t)\,&=\partial_{ss}^2 \kappa_r(E_t,x_t)+\kappa^2(E_t,x_t)\,
\kappa_r(E_t,x_t)\\
&= \frac{\partial^2_{ss}\kappa(E_t,x_t)}2+\kappa^2(E_t,x_t)\,
\kappa_r(E_t,x_t).
\end{split}
\end{equation*}
Hence, if~$x_t^\star$ is minimal for~$\kappa(E_t,\cdot)$, we have that
$$
\partial^2_{ss}\kappa(E_t,x_t^\star)
\ge0$$
and
$$ \partial_t \kappa(E_t,x_t)\ge \kappa^2(E_t,x_t)\kappa_r(E_t,x_t)\ge0.$$
This gives that~$\kappa(E_t,\cdot)$ is nondecreasing at the minimal points,
and thus nonnegative, which completes the proof of Theorem \ref{TH3}.

\section{Proof of Theorem \ref{TH4}}\label{PTH4}

\begin{figure}
    \centering
    \includegraphics[width=9cm]{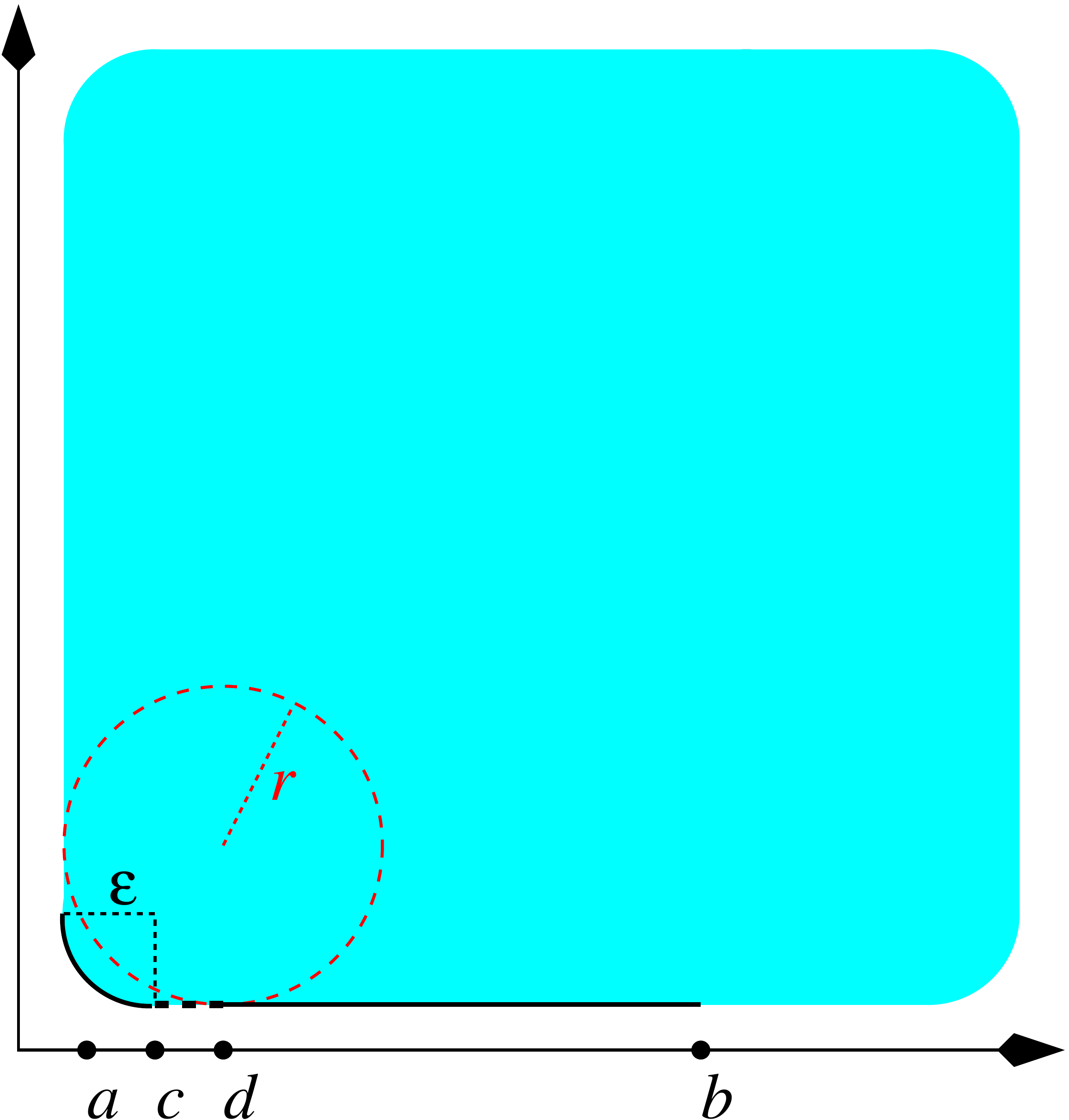}
    \caption{\em {{The set in the proof of Theorem~\ref{TH4}.}}}
    \label{W823495-123I77-BIS}
\end{figure}

We construct a convex set~$E$ as depicted in Figure~\ref{W823495-123I77-BIS}.
Namely, we consider the smoothing of a square
of side~$1>> r$
in which the corners are rounded by a curve with curvature of order~$\e<< r$.
Notice that the curvature of~$\partial E$
vanishes along the dashed and solid line on the bottom of
Figure~\ref{W823495-123I77-BIS}, and it is of the order of~$1/\e$
along the solid arc. As for the $r$-curvature, from~\eqref{KAPPA}
we see that~$\kappa_r$ is equal to~$\kappa$ along 
the solid line and to~$\frac\kappa2+\frac1{2r}$
along the dashed line and the solid arc.
That is,~$\kappa_r$ is equal to~$0$ along 
the solid line, equal to~$\frac1{2r}$
along the dashed line
and of order~$\frac1\e$ along\footnote{It is interesting to remark
that Figure~\ref{W823495-123I77-BIS} well explains
the difference between the $r$-curvature and the classical one
for convex sets.
Namely, along the solid line, we are on a ``large scale'' and
the $r$-curvature coincides with the classical one.
Then, on a scale of order~$r$, given by the dashed line,
the $r$-curvature becomes greater than the classical one.
But on a very small scale, 
the $r$-curvature may become smaller than the classical one,
since, on 
the solid arc,
we have that~$\kappa_r=
\frac\kappa2+\frac1{2r} <\kappa$, since~$\kappa\sim\frac1\e\gg\frac1{r}$.}
the solid arc.

Consequently, if we put Cartesian axes as in Figure~\ref{W823495-123I77-BIS},
we can describe~$\kappa_r$ along the bottom of the set~$E$ by a function~$\varphi$,
and,
considering intervals~$(a,b)\supset(c,d)$ as in 
Figure~\ref{W823495-123I77-BIS},
it holds that~$\varphi=\frac1{2r}$ on~$(c,d)$
and~$\varphi=0$ on~$(d,b)$. In particular,
\begin{equation}\label{NOTC}
{\mbox{$\varphi$ is not convex.}}
\end{equation}
Now, to establish Theorem~\ref{TH4} we argue by contradiction
and suppose that the set~$E$ evolves in time~$t\in[0,T)$
into a convex set~$E_t$.
We point out that the set
\begin{equation}\label{9GRAP}
{\mbox{$\partial E_t\cap\big(\R\times(-h,h)\big)$
is a graph in the vertical direction for all~$t\in[0,T)$,}}\end{equation} provided that~$h$, $T>0$
are chosen appropriately small.
Indeed, if not, a vertical segment would meet~$\partial E_t\cap\big(\R\times(-h,h)\big)$
at least twice; then, since another intersection with~$\partial E_t$ must occur
close to the top, this would contradict the convexity of~$E_t$.

Hence, as a consequence of~\eqref{9GRAP},
the boundary of the set~$E_t$
can be locally described, in the vicinity of the interval~$(a,b)$
by a convex
function~$\Psi=\Psi(x,t)$. By~\eqref{FLOW}
and the fact that
\begin{equation}\label{92013l10}
{\mbox{$\Psi(x,0)$ is constant for all~$x\in(c,b)$,}}
\end{equation}
we know that the velocity of the flow at time~$0$ coincides
with~$\partial_t\Psi(x,0)$, hence
\begin{equation}\label{901-2-1e2}
\partial_t\Psi(x,0)=\varphi(x)\quad{\mbox{ for all~$x\in(c,b)$}}.\end{equation}
Also, since~$\Psi$ is convex,
for any~$x$, $y\in(c,b)$, any~$\vartheta\in[0,1]$ and any~$t\in(0,T)$
we have that
\begin{equation*} 
\Psi\big((1-\vartheta)x+\vartheta y,t\big)\le
(1-\vartheta)\Psi(x,t)+\vartheta\Psi(y,t)
,\end{equation*}
and therefore, in view of~\eqref{92013l10}
and~\eqref{901-2-1e2},
\begin{eqnarray*}
\varphi( (1-\vartheta)x+\vartheta y )&=&
\partial_t
\Psi\big((1-\vartheta)x+\vartheta y,0\big)\\&
=&
\lim_{t\searrow0}
\frac{\Psi\big((1-\vartheta)x+\vartheta y,t\big)-
\Psi\big((1-\vartheta)x+\vartheta y,0\big)}{t}
\\&\le&\lim_{t\searrow0}\frac{
(1-\vartheta)\Psi(x,t)+\vartheta\Psi(y,t)}t-
\frac{
\Psi\big((1-\vartheta)x+\vartheta y,0\big)}{t}
\\&=&
\lim_{t\searrow0}\frac{
(1-\vartheta)\Psi(x,t)+\vartheta\Psi(y,t)}t-
\frac{(1-\vartheta)
\Psi\big(x,0\big)
+\vartheta
\Psi\big(y,0\big)
}{t}
\\&
=&\lim_{t\searrow0}\frac{
(1-\vartheta)\big(\Psi(x,t)-\Psi(x,0)\big)
+\vartheta\big(\Psi(y,t)-\Psi(y,0)\big)}t
\\&=&
(1-\vartheta)\partial_t
\Psi(x,0)+\vartheta\partial_t
\Psi(y,0)\\&
=&(1-\vartheta)\varphi(x)+\vartheta\varphi(y).
\end{eqnarray*}
This gives that~$\varphi$ is convex in~$(c,b)$,
which is a contradiction with~\eqref{NOTC} and so Theorem~\ref{TH4}
is proved.

\section{Proof of Theorem~\ref{TW}}\label{TWA}

Without loss of generality,
we can normalize the speed~$c$ to be equal to~$1$.
Such dilation, in the new coordinate frame,
transforms the condition~$c\in (0,r)$
into
\begin{equation}\label{r1}
r>1.\end{equation}
Then, if~$h$ is a traveling wave as in~\eqref{TWW} with~$c=1$ for a geometric flow
with inner normal velocity~$v$, one sees that~$h$ is a solution of
\begin{equation}\label{VB-1}
v(x)\,\sqrt{1+|h'(x)|^2}=1.
\end{equation}
In particular, for the classical mean curvature flow, we have that
$$v(x)=\frac{h''(x)}{
(1+|h'(x)|^2)^{3/2}},$$
and so~\eqref{VB-1} becomes
\begin{equation}\label{VB-2}
\frac{h''(x)}{1+|h'(x)|^2}=1.
\end{equation}
Similarly, in the regime in which~$\kappa_r=\frac\kappa2+\frac1{2r}$,
the flow in~\eqref{FLOW} and~\eqref{VB-1} yield the equation
\begin{equation}\label{VB-3}
\frac{h''(x)}{1+|h'(x)|^2}=2-\frac{\sqrt{1+|h'(x)|^2}}{r}.
\end{equation}
Our objective is now to consider a (suitable translation of a)
solution~$h_\infty$
of~\eqref{VB-2} (far from the origin)
and glue it to a solution~$h_0$ of~\eqref{VB-3} (near the origin).
The joint will be done in such a way that the final curve is~$C^{1,1}$
(roughly speaking, the building arcs will share the tangent
line at the matching point).

To implement this construction, we consider the Cauchy problem
\begin{equation}\label{SDXx} \left\{ \begin{matrix}
\phi'(x)=2(1+|\phi(x)|^2)-\displaystyle\frac1r\,(1+|\phi(x)|^2)^{3/2},\\
\phi(0)=0.\end{matrix}
\right.\end{equation}
The solution to this problem exists (and it is
unique) for small values of~$x$,
and we extend it to its largest existence interval~$(x_-,x_+)$,
with~$-\infty\le x_-<0<x_+\le+\infty$.
We claim that
\begin{equation}\label{SDX}
{\mbox{for all~$x\in(x_-,x_+)$,
we have that~$\phi(x)\le\sqrt{4r^2-1}$}}.\end{equation}
Indeed, suppose, by contradiction, that there exists~$\bar x\in(x_-,x_+)$
such that~$\phi(\bar x)>\sqrt{4r^2-1}$. Then, for any~$\e>0$ sufficiently
small, there exists~${\bar x_\e}$ on the segment joining~$0$ to~$\bar x$
such that~$\phi({\bar x_\e})=\sqrt{4r^2(1+\e)-1}$ with~$\phi'({\bar x_\e})\ge0$.
Then, we have that
\begin{eqnarray*}
0 &\le&
\phi'({\bar x_\e})\\&=&
2(1+|\phi({\bar x_\e})|^2)-\frac1r\,(1+|\phi({\bar x_\e})|^2)^{3/2}\\
&=& 2(4r^2(1+\e))-\frac1r\,(4r^2(1+\e))^{3/2}\\&=&
8r^2(1+\e)-8r^2(1+\e)^{3/2}\\&=&
8r^2(1+\e)(1-\sqrt{1+\e})
\\ &<&0,
\end{eqnarray*}
which is a contradiction, proving~\eqref{SDX}.

We also have that
\begin{equation}\label{SDX2}
{\mbox{for all~$x\in(x_-,x_+)$,
it holds that~$\phi(x)\ge0$}}.\end{equation}
Indeed, suppose, by contradiction, that
for any~$\e>0$ small enough there exists~$\tilde x_\e\in(x_-,x_+)$
such that~$\phi({\tilde x_\e})=-\e$ with~$\phi'({\tilde x_\e})\le0$.
Then, it holds that
\begin{eqnarray*}
0 &\ge&
\phi'({\tilde x_\e})\\&=&
2(1+|\phi({\tilde x_\e})|^2)-\frac1r\,(1+|\phi({\tilde x_\e})|^2)^{3/2}\\
&=&
2(1+\e^2)-\frac1r\,(1+\e^2)^{3/2}.
\end{eqnarray*}
Hence, taking~$\e$ arbitrarily small, we obtain that~$0\ge2-\frac1r$,
which gives a contradiction with~\eqref{r1}, thus proving~\eqref{SDX2}.

As a consequence of~\eqref{SDX} and~\eqref{SDX2}, we obtain that~$x_+=+\infty$ and~$x_-=-\infty$,
and so~$\phi$ is a global solution of~\eqref{SDXx}. In addition,
since~$x\mapsto -\phi(-x)$ is also a solution of~\eqref{SDXx},
by the uniqueness result
of the Cauchy problem we obtain that~$\phi(x)=-\phi(-x)$,
i.e.
\begin{equation}\label{odd111}
{\mbox{$\phi$ is odd.}}\end{equation}
Moreover, from~\eqref{SDX}, we have
that
$$2-\frac1r\sqrt{1+|\phi(x)|^2}\ge0$$ and so, by~\eqref{SDXx},
$$ \phi'(x)=(1+|\phi(x)|^2)\left(2-\displaystyle\frac1r\sqrt{1+|\phi(x)|^2}\right)\ge0.$$
Accordingly, 
\begin{equation}\label{MONPH1}
{\mbox{$\phi$ is monotone nondecreasing}}\end{equation}
and so the following limit exists
$$ \ell:=\lim_{x\to+\infty}\phi(x).$$
Also, by~\eqref{SDXx}, \eqref{SDX} and~\eqref{MONPH1}, it holds that
$$ 0=\lim_{x\to+\infty}\phi'(x)=
2(1+\ell^2)-\frac1r\,(1+\ell^2)^{3/2},$$
and therefore
\begin{equation*}
\lim_{x\to+\infty}\phi(x)=\ell=\sqrt{4r^2-1}.\end{equation*}
We also define
\begin{equation}\label{hzerosol}
h_0(x):=\int_0^x \phi(\xi)\,d\xi.\end{equation}
Notice that~$h_0$ is convex, in view of the monotonicity of~$\phi$,
and it is even, due to~\eqref{odd111}.
Also, by~\eqref{SDXx}, we know that
\begin{equation*} \left\{ \begin{matrix}h_0''(x)=2(1+
|h_0'(x)|^2)-\displaystyle\frac1r\,(1+|h_0'(x)|^2)^{3/2},\\
h_0(0)=0=h_0'(0),\end{matrix}
\right.\end{equation*}
and so~$h_0$ is a solution of~\eqref{VB-3}.

\begin{figure}
    \centering
    \includegraphics[width=12cm]{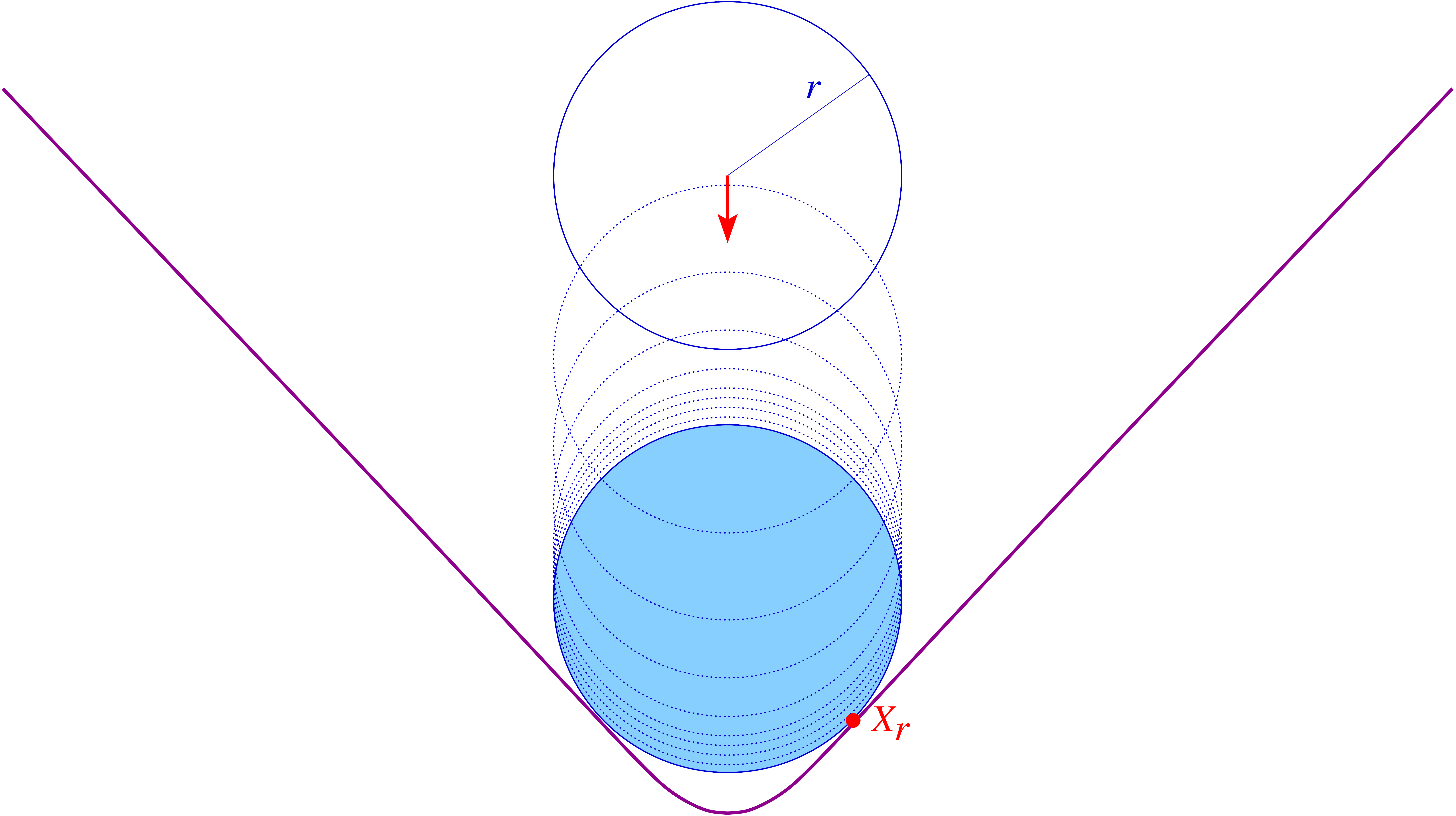}
    \caption{\em {{Dropping balls inside the supergraph of~$h_0$.}}}
    \label{lala}
\end{figure}

The idea is now to use a ``dropping the ball in the basket'' method, see Figure~\ref{lala}.
Namely, we slide a ball~$B_r(te_n)$ from $t=+\infty$ downwards, till
it touches the graph of~$h_0$. We denote by~$X_r=(x_r,h_0(x_r))$ a touching point
situated on the right branch
of the graph of~$h_0$. 
We also denote
the curvature of the supergraph of~$h_0$ at a point~$X=(x,h_0(x))$ by~$\kappa_0(x)$.
Since the supergraph of~$h_0$ contains a ball of radius~$r$
tangent at~$x_r$, we have that
\begin{equation}\label{AK-1-1}
\kappa_0(x_r)\le\frac1r.
\end{equation}
On the other hand, by~\eqref{VB-3},
\begin{equation}\label{AK-1-2}
\kappa_0(x)
=\frac{h''_0(x)}{(1+|h'_0(x)|^2)^{3/2}}=\frac1{\sqrt{1+|h'_0(x)|^2}}\left(
2-\frac{\sqrt{1+|h'_0(x)|^2}}{r}\right)=\frac2{\sqrt{1+|h'_0(x)|^2}}-\frac1r.
\end{equation}
By~\eqref{AK-1-1} and~\eqref{AK-1-2}, it follows that
$$ \frac1r\ge \kappa_0(x_r)=\frac2{\sqrt{1+|h'_0(x_r)|^2}}-\frac1r$$
and therefore
\begin{equation}\label{MAsIS}
h_0'(x_r)\ge\sqrt{r^2-1}.
\end{equation}
Now we consider the ``standard grim reaper''
\begin{equation}\label{GRIM}\left[-\frac\pi2,\frac\pi2\right]\ni x\mapsto
h_\infty(x)=-\log(\cos x),\end{equation}
and we define
\begin{equation}\label{jfbbfbg}
\tilde x_r:=\arctan h_0'(x_r).
\end{equation}
Notice that,
in view of~\eqref{r1},
we have that~$\tilde x_r
\in\left(0,\frac\pi2\right)$.
We also set
\begin{equation}\label{hstar} h_\star(x):=
\left\{\begin{matrix}
h_0(x) & {\mbox{ if }} |x|\le x_r,\\
h_\infty(x+\tilde x_r-x_r)-h_\infty(\tilde x_r)+h_0(x_r)
& {\mbox{ if }} x\in\left( x_r, \,\displaystyle\frac\pi2+x_r-\tilde x_r\right),\\
h_\infty(x-\tilde x_r+x_r)-h_\infty(\tilde x_r)+h_0(x_r)
& {\mbox{ if }} x\in\left(-
\displaystyle\frac\pi2+\tilde x_r- x_r
,\,- x_r\right).
\end{matrix}
\right.\end{equation}
The function~$h_\star$ is depicted in Figure~\ref{NEWG}.
By inspection, we see that~$h_\star$ is continuous and even.
Also, by \eqref{GRIM} and~\eqref{jfbbfbg},
we have that
$$ h_\infty'(\tilde x_r)=\tan\tilde x_r=h_0'(x_r),$$
and so~$h_\star\in C^{1,1}(\R)$. 

\begin{figure}
    \centering
    \includegraphics[width=12cm]{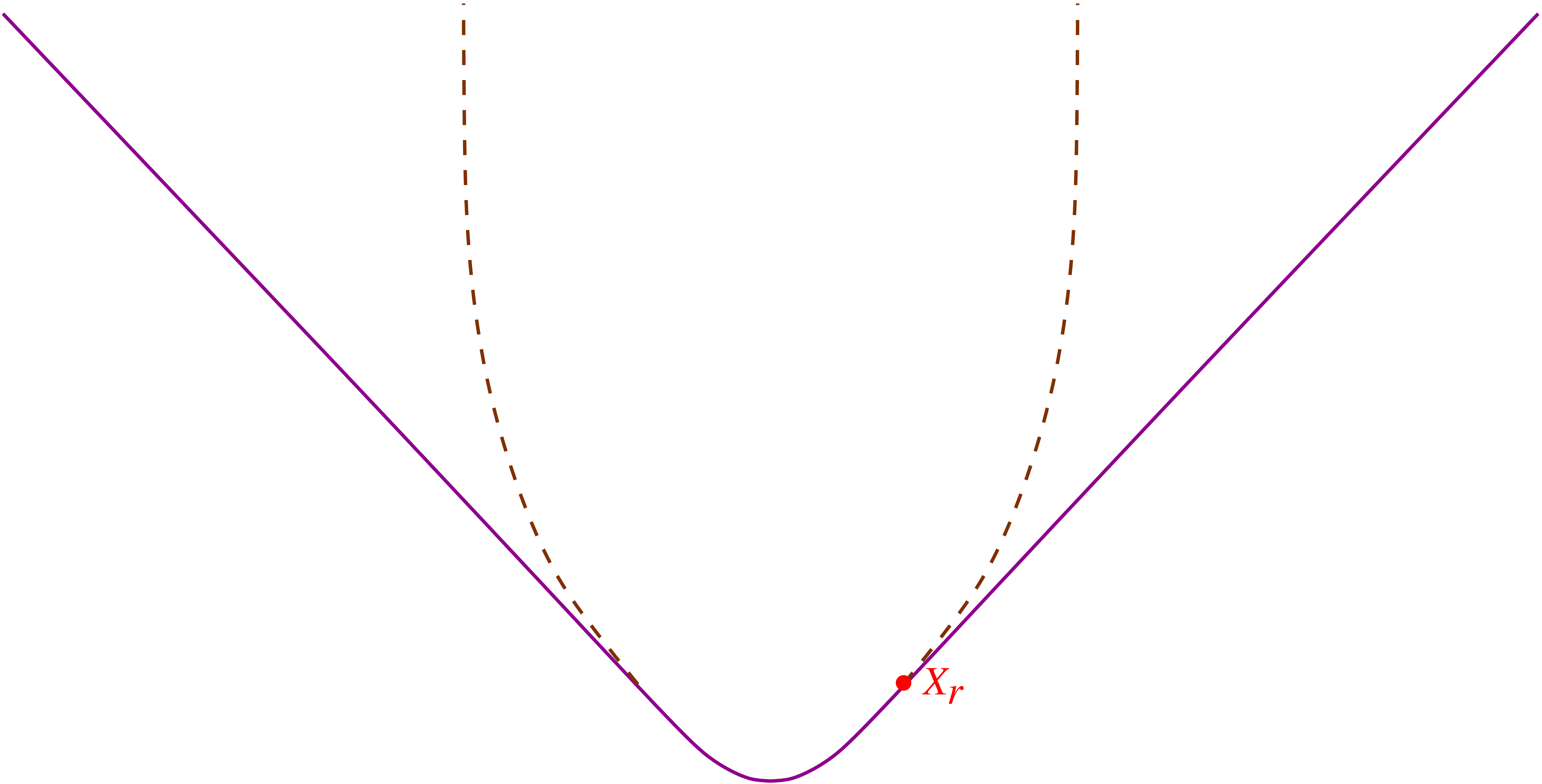}
    \caption{\em {{The new traveling wave in~\eqref{hstar}.
    The dashed curves are branches of classical grim reapers. The solid 
    curve represents the graph of~$h_0$.}}}
    \label{NEWG}
\end{figure}

It is also useful to observe that
\begin{equation}\label{UNIQUE}
{\mbox{the touching point~$X_r$ on the right branch of~$h_0$ is unique.}}\end{equation}
Indeed, 
from~\eqref{AK-1-2} and the monotonicity of~$h_0'=\phi$ it follows
that~$\kappa_0(x)\le\kappa_0(x_r)$ for any~$x>x_r$. This and~\eqref{AK-1-1}
yield that~$\kappa_0(x)\le 1/r$ for any~$x>x_r$, from which it follows that
there cannot be another touching point on the right branch of~$h_0$
with~$x>x_r$. On the other hand, if there was another touching point~$\tilde{X}_r$
with~$\tilde{x}_r<x_r$, then this observation would also imply that
there cannot be a touching point on the right
branch of~$h_0$ with~$x>\tilde{x}_r$, which gives a contradiction
and completes the proof of~\eqref{UNIQUE}.

We observe that, by~\eqref{r1} and~\eqref{AK-1-2},
$$ \kappa_0(0)=\frac2{\sqrt{1+|h'_0(0)|^2}}-\frac1r=2-\frac1r>1>\frac1r
$$
and this implies that the touching point cannot occur at the origin, namely~$x_r>0$.

\begin{figure}
    \centering
    \includegraphics[width=12cm]{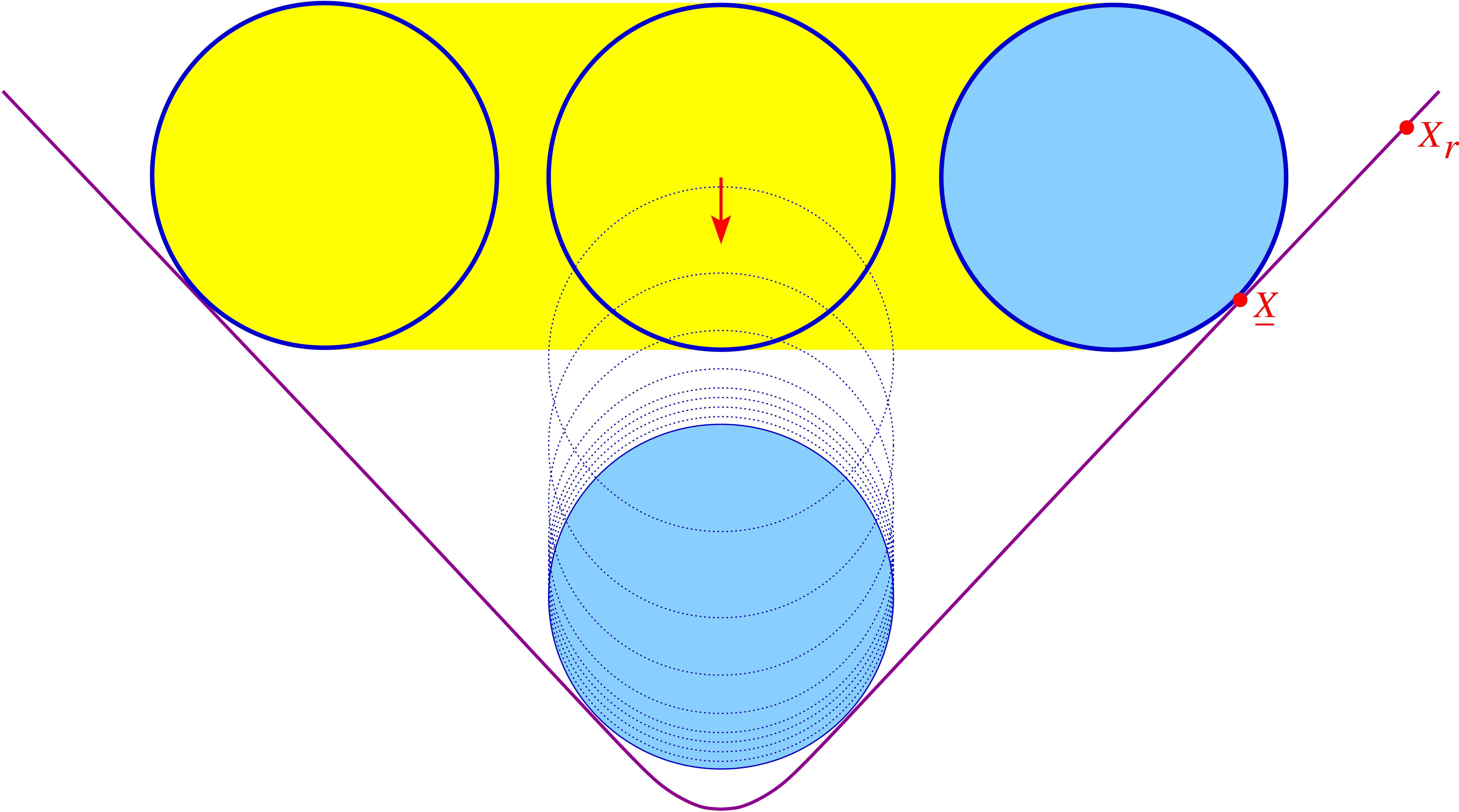}
    \caption{\em {{Proof of~\eqref{NPD}.}}}
    \label{lala2}
\end{figure}

Furthermore, 
\begin{equation}\label{NPD}
\begin{split}&
{\mbox{the graph of~$h_0$ cannot be touched from above}}\\&
{\mbox{by a ball of radius~$r$
at any point~$X=(x,h_0(x))$ with~$|x|<x_r$.}} \end{split}
\end{equation}
Indeed, suppose by contradiction that
the graph of~$h_0$ is touched from above by~$B_r(\underline p_1,\underline p_2)$
at a point~$\underline X=(\underline x,h_0(\underline x))$ with~$0< \underline x<x_r$.
Since~$h_0$ is even, it follows that~$B_r(-\underline p_1,\underline p_2)$
is also contained in the supergraph of~$h_0$. Hence, by convexity,
the ball~$B_r(0,\underline p_2)$
is contained in the supergraph of~$h_0$. Dropping down such ball, we obtain
a touching point~$\hat X=(\hat x,h_0(\hat x))$ with~$0<\hat x<x_r$,
thus producing a contradiction with~\eqref{UNIQUE}
and completing the proof of~\eqref{NPD} (see Figure~\ref{lala2} for a sketch
of this argument).

To complete the proof of Theorem~\ref{TW},
we now check that the translating supergraph of~$h_\star$ is a solution
of the geometric flow in~\eqref{FLOW}. For this, since this supergraph
is convex, and~$h_\infty$ and~$h_0$ are solutions of~\eqref{VB-2}
and~\eqref{VB-3}, respectively, recalling~\eqref{NPD}
it is enough to check that 
\begin{equation}\label{CURSPx}\begin{split}
&{\mbox{the curvature of the supergraph is
smaller than~$\displaystyle\frac1r$ when~$|x|\in\left(x_r,\,\displaystyle\frac\pi2+x_r-\tilde x_r\right)$.}}\end{split}\end{equation}
To this aim, 
we observe that when~$x\in\left(x_r,\,\frac\pi2+x_r-\tilde x_r\right)$
we have that~$x+\tilde x_r-x_r\in\left(\tilde x_r,\,\frac\pi2\right)\subseteq
\left(0,\,\frac\pi2\right)$,
and thus we deduce from~\eqref{MAsIS}, \eqref{GRIM} and~\eqref{jfbbfbg} that
\begin{eqnarray*}&& \frac{h_\star''(x)}{(1+|h_\star'(x)|^2)^{3/2}}=
\frac{h_\infty''(x+\tilde x_r-x_r)}{(1+|h_\infty'(x+\tilde x_r-x_r)|^2)^{3/2}}=
|\cos(x+\tilde x_r-x_r)|\\
&&\qquad=\cos(x+\tilde x_r-x_r)\le\cos(\tilde x_r)=
\cos(\arctan h_0'(x_r))\le
\cos(\arctan\sqrt{r^2-1})=\frac1r.
\end{eqnarray*}
This establishes~\eqref{CURSPx}
when~$x\in\left(x_r,\,\frac\pi2+x_r-\tilde x_r\right)$,
and the case~$x\in\left(-\frac\pi2+\tilde x_r- x_r,\,-x_r \right)$
is symmetric. Hence, the proof of~\eqref{CURSPx}, and so
of Theorem~\ref{TW}, is complete.

\begin{remark} {\normalfont
We observe that the translating solution constructed in Theorem~\ref{TW}
is only~$C^{1,1}$ and not~$C^2$ in principle. On the other hand, such regularity
is enough to define the velocity field at any point in a continuous way,
thanks to~\eqref{VB-1}.}
\end{remark}

\vfill

{\footnotesize

\noindent {\em Addresses:} \\

Serena Dipierro \& Enrico Valdinoci.
Department of Mathematics
and Statistics,
University of Western Australia,
35 Stirling Hwy, Crawley WA 6009, Australia,
and
Dipartimento di Matematica, Universit\`a di Milano,
Via Saldini 50, 20133 Milan, Italy.

Matteo Novaga.
Dipartimento di Matematica,
Universit\`a di Pisa, 
Largo B. Pontecorvo 5, 56127 Pisa,
Italy.\\

Enrico Valdinoci.
School of Mathematics
and Statistics,
University of Melbourne, 813 Swanston St,
Parkville VIC 3010, Australia, and
IMATI-CNR, Via Ferrata 1, 27100 Pavia,
Italy. \\
\smallskip

\noindent{\em Emails:}\\

{serena.dipierro@unimi.it},
{matteo.novaga@unipi.it}, {enrico@mat.uniroma3.it}  

}

\end{document}